\documentclass{article}[12pt,a4paper,fullpage]  
\usepackage{hyperref}
\usepackage{url}
\usepackage[utf8]{inputenc}
\usepackage{amsthm,amsmath,amssymb,amsfonts,latexsym}
\usepackage{enumitem}
\usepackage{standalone}
\usepackage{mathtools}
\usepackage{pdfpages}
\usepackage[english]{babel}
\usepackage{gensymb}
\usepackage{tabularx}
\usepackage{multirow}
\usepackage[left=3cm, right=3cm, top=3cm, bottom=4cm]{geometry}

\usepackage{amsfonts,latexsym,amsmath}
\usepackage{algorithmic}
\usepackage{abbr}

\numberwithin{equation}{section}

\theoremstyle{plain}% default
\newtheorem{theorem}{Theorem}[section]

\newtheorem{proposition}[theorem]{Proposition}
\newtheorem{algorithm}  [theorem]{Algorithm}

\theoremstyle{definition}
\newtheorem{definition}{Definition}[section]
\newtheorem{assumption}{Assumption}[section]

\theoremstyle{remark}
\newtheorem{remark}{Remark}[section]

\usepackage{graphicx,color}
\usepackage{amsfonts,latexsym,amsmath}
\usepackage{algorithm,algorithmic}

\newcommand{\bl}{\begin{list}{ \ }{
\leftmargin=.325in}}
\newtheorem{prop}{Proposition}[section]
\newenvironment{thm*}[2]{\par\bgroup{\scshape #1 \rm(#2).}
\it\ignorespaces}{\egroup}
\newtheorem{stopping rule}[theorem]{Stopping Rule}
\newcommand{\balg}{\begin{algorithm}}
\newcommand{\ealg}{\end{algorithm}}

\newcommand{\br}{{\bf r}}
\newcommand{\Om}{{\bf \Omega}}
\newcommand{\al}{{\boldsymbol{ \alpha}}}
\newcommand{\bPhi}{{\boldsymbol \Phi}}

\newcommand{\bphi}{{\boldsymbol \varphi}}
\newcommand{\bA}{{\bf A}}

\def\data#1{\varphi_{\alpha_{#1}}}
\def\layer#1#2{\Phi^{(#1)}(\br+h_{#1}\al_{#2})}
\def\layerj#1{\Phi^{(#1)}}

\def\Ao#1{{\bf A}_{\al_{#1}}}

\def\tlayerjrec#1{{\tilde\Phi}^{(#1)}}

\def\spr#1{\langle\, #1 \,\rangle}

\def\bOmgl#1#2{\bar\Omega^{#1,#2}}
\def\bOmgT{\bar\Omega_T^g}

\begin{document}
\title{On some analytic properties of the atmospheric tomography operator: Non-Uniqueness and reconstructability issues}
 \author{
 Ronny Ramlau\thanks{Industrial Mathematics Institute, Johannes Kepler
   University Linz, and Johann Radon Institute for Computational and Applied
   Mathematics, A-4040 Linz, Austria (ronny.ramlau@jku.at).} 
   \and Bernadett Stadler\thanks{Johann Radon Institute for Computational and Applied
   Mathematics, A-4040 Linz, Austria (bernadett.stadler@ricam.oeaw.ac.at).}}
\maketitle

\begin{abstract}
In this paper, we consider the atmospheric tomography operator, which describes the effect of turbulent atmospheric layers on incoming planar wavefronts. Given wavefronts from different guide stars, measured at a telescope, the inverse problem consists in the reconstruction of the turbulence above the telescope. We show that the available data is not sufficient to reconstruct the atmosphere uniquely. Additionally, we show that classical regularization methods as Tikhonov regularization or Landweber iteration will always fail to reconstruct a physically meaningful turbulence distribution.
\end{abstract}

\noindent \textbf{Keywords:} Atmospheric tomography operator, ill-posed and ill-conditioned problems\\

\noindent \textbf{AMS:} 78A10, 47A52

\section{Introduction} \label{intro}
Images from earth bound telescopes are obtained by looking at the sky through the atmosphere. Inhomogenious distributed air patches with different temperature have different refractive indices, which causes the wavefronts of the incoming light from astronomical objects to travel with locally varying speed. This effect reduces the image quality of the obtained astronomical images and worsens for larger telescopes. As a counter-measure, modern giant telescopes are equipped with {\it Adaptive Optics} systems, which use a combination of wavefront sensors and deformable mirrors for the correction of the incoming
disturbed wavefronts \cite{Roddier,RoWe96}. Based on wavefront sensor measurements of distant guide stars, appropriate shapes of deformable mirrors are obtained such that the blurring caused by the turbulence of the atmosphere is removed after reflection, see Figure \ref{defmirr} (left).

\bfig[ht]
\bmp{0.6\textwidth}
\includegraphics[width=0.9\textwidth]{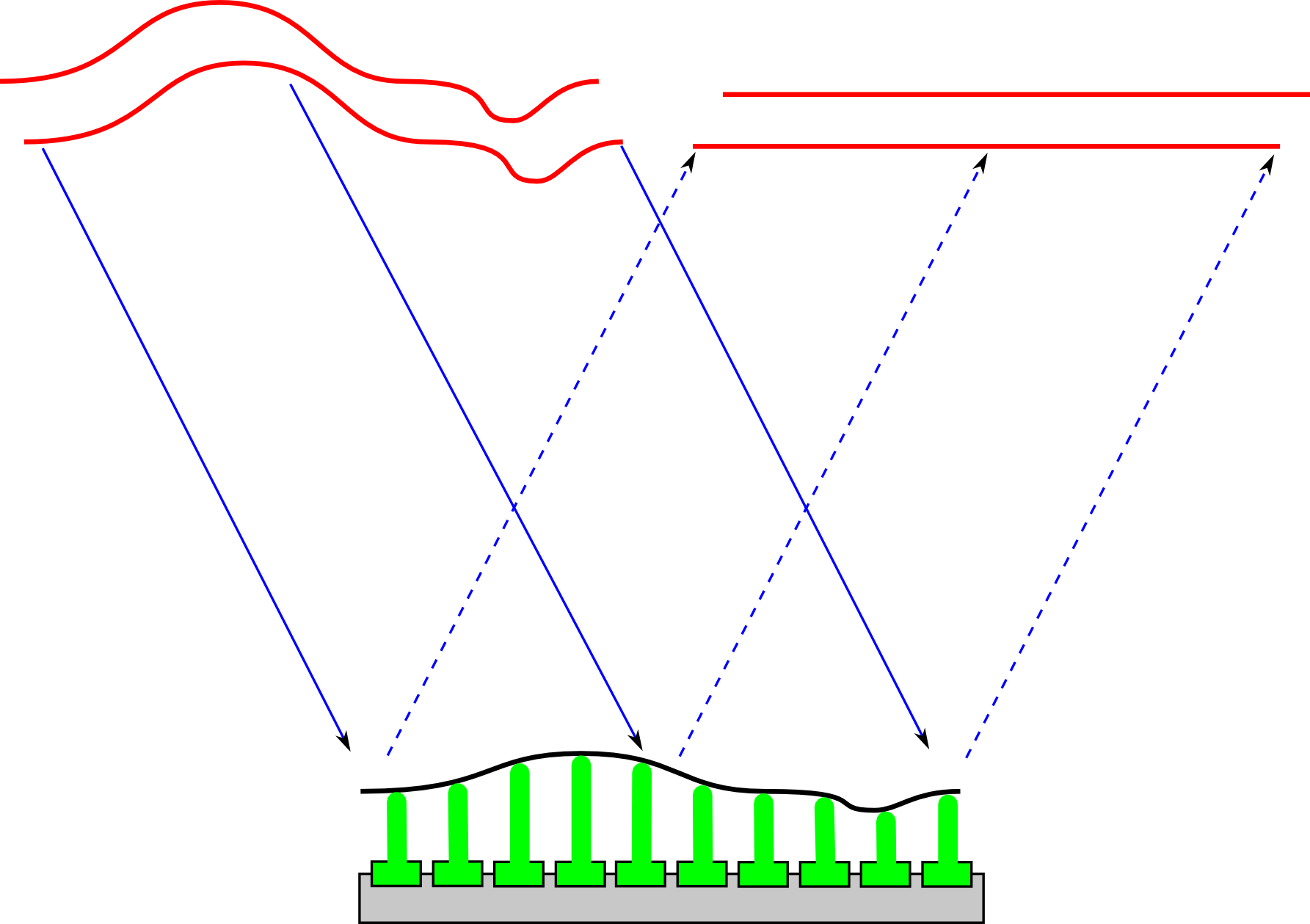}
\emp\bmp{0.4\textwidth}
\includegraphics[width=0.9\textwidth]{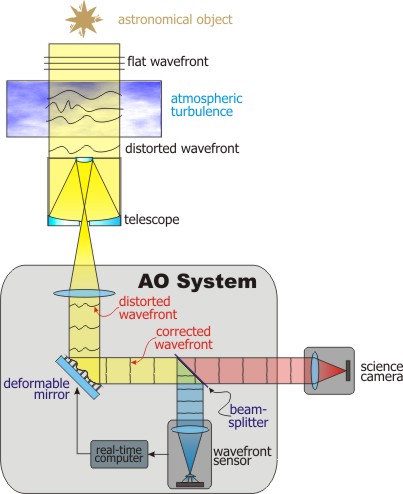}
\emp
\caption{\label{defmirr}Left: correction of an incoming wavefront by a
 deformable mirror (image from \cite{Au17}); Right: sketch of a SCAO
 system (image from \cite{Dhil16}).}
\efig

{\em Single Conjugate Adaptive Optics} (SCAO) is the simplest Adaptive Optics system: The light of a single bright {\it Natural Guide Star (NGS)} near to the astronomical object of interest is used to compute the optimal correcting shape of a single deformable mirror. The general principle of an SCAO system is given in Figure \ref{defmirr} (right).

More complex AO systems use the incoming light of several guide stars, measured
by several wavefront sensors, and multiple deformable mirrors to achieve either
a high image quality over a large field of view ({\em Multi Conjugate Adaptive
Optics}, MCAO) or to achieve high image quality in prescribed directions
({\em Laser Tomography Adaptive Optics}, LTAO, and {\em Multi Object Adaptive
Optics}, MOAO). As there are not enough bright guide stars available at the night sky, artificially created {\it Laser guide stars (LGS)} are commonly used.
The computation of the mirror shapes for MCAO, LTAO and MOAO is based on the
reconstruction of the turbulence profile above the telescope from the incoming
wavefront measurements, which is obtained by solving the Atmospheric Tomography
problem. As the  separation of the guide star is low ($1$ arcmin for
MCAO, $3.5$ arcmin for MOAO) we are facing a limited angle tomography problem
which is known to be severely ill posed \cite{Davison83,Nat86}. Moreover, only few data are available, as only a small number of guide stars can be used (e.g., $6$ laser guide stars for the {\it Extremely Large Telescope (ELT)} of the {\it European Southern Observatory (ESO)}. Thus it is hopeless to attempt to reconstruct the full turbulent volume above the telescope. Instead, we use the
common assumption that the atmosphere contains only a limited number of turbulent layers, which are infinitely thin and located at a prescribed height, see Figure \ref{MCAOSys} for a related MCAO setup.

\begin{figure}[ht!]
\cl{\includegraphics[width=0.7\textwidth]{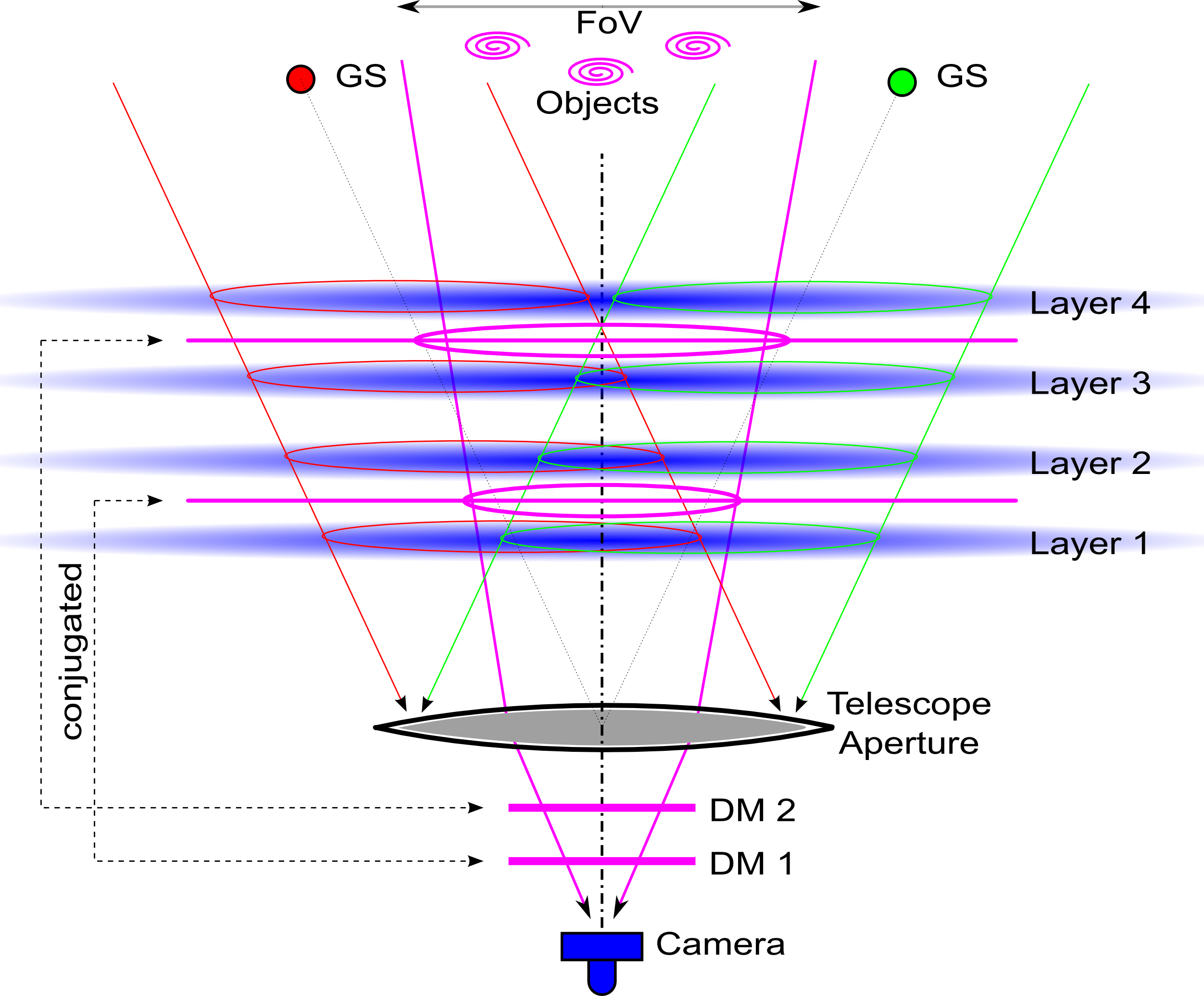}}
\caption{\label{MCAOSys}Sketch of an MCAO system assuming a four layer
 atmosphere and using two deformable mirrors for correction (image from
 \cite{Au17}).}
\end{figure}

Atmospheric Tomography has been in the focus of astronomers and AO engineers for years. Its mathematical description can be derived from the standard Radon
transform by taking into account the geometry of the telescope, the limited
angle property and the layered structure of the atmosphere \cite{ElVo09,Fusco}. Most of the research has been dedicated to the development of
numerical reconstruction approaches: e.g., Fusco et.al. \cite{Fusco} propose a minimum mean square
error method, while in \cite{Gavel2004} a back-projection algorithm has been used. In particular popular are conjugate gradient type iterative reconstructors with suitable preconditioning \cite{ElGiVo03,GiEl08,GiVoEl02,YaVoEl06,VoYa06a}. The Fractal Iterative Method (FrIM) \cite{Tallon_et_al_10,ThiTa10,TaBeTaLeThClMa11} and the Finite Element Wavelet Hybrid Algorithm (FEWHA) \cite{Yu14,HeYu13,YuHeRa13,YuHeRa13b} both compute maximum a-posterior estimators for the solution of the atmospheric tomography problem, whereas a Kaczmarz iteration is used in \cite{RaRo12,RoRa13}. In infinite dimensional function spaces, the atmospheric tomography operator has been considered in \cite{RaRo12,EslPechRam} with a focus on the derivation of the adjoints of the operator with respect to different underlying function spaces. In a standard $L_2$ setting, singular value type decompositions \cite{NeubRamlau_2017} as well as frame decompositions \cite{RamlauHub_2020,Ramlau_2021_02}   have been derived. The singular value type decomposition already suggest the existence of a nonzero nullspace of the atmospheric tomography operator, as some of the "singular values" are numerically zero, but a throughout analysis is still missing. Additionally, the used reconstruction methods all seem to fail to reconstruct the original atmosphere, which is another hint towards a nontrivial nullspace.\\

The aim of our paper is therefore to gain some insight into the structure of the Atmospheric Tomography problem. In particular, we will show that a proper reconstruction in {\it non-overlapping} areas is impossible.
Additionally, we will investigate the structure of the reconstructions obtained by standard regularization methods which use the adjoint of the atmospheric tomography operator. We will show that those methods almost always fail to fully reconstruct a realistic atmosphere. The analytical results are supported by numerical simulations.\\

The paper is organized as follows: in Section \ref{setting} we define the geometrical setup
and define the related atmospheric tomography operator and its adjoint. In Section \ref{nonuniqueness} we show that the atmospheric tomography operator is not uniquely invertible. Section \ref{reconstructability} analyses the structure of solutions obtained by classical reconstruction methods like Tikhonov regularization, Landweber iteration or the Kaczmarz method. Finally we verify our results numerically in Section \ref{numerics}.

\section{Setting}\label{setting}

\bfig[ht]
\cl{\includegraphics[width=0.6\textwidth]{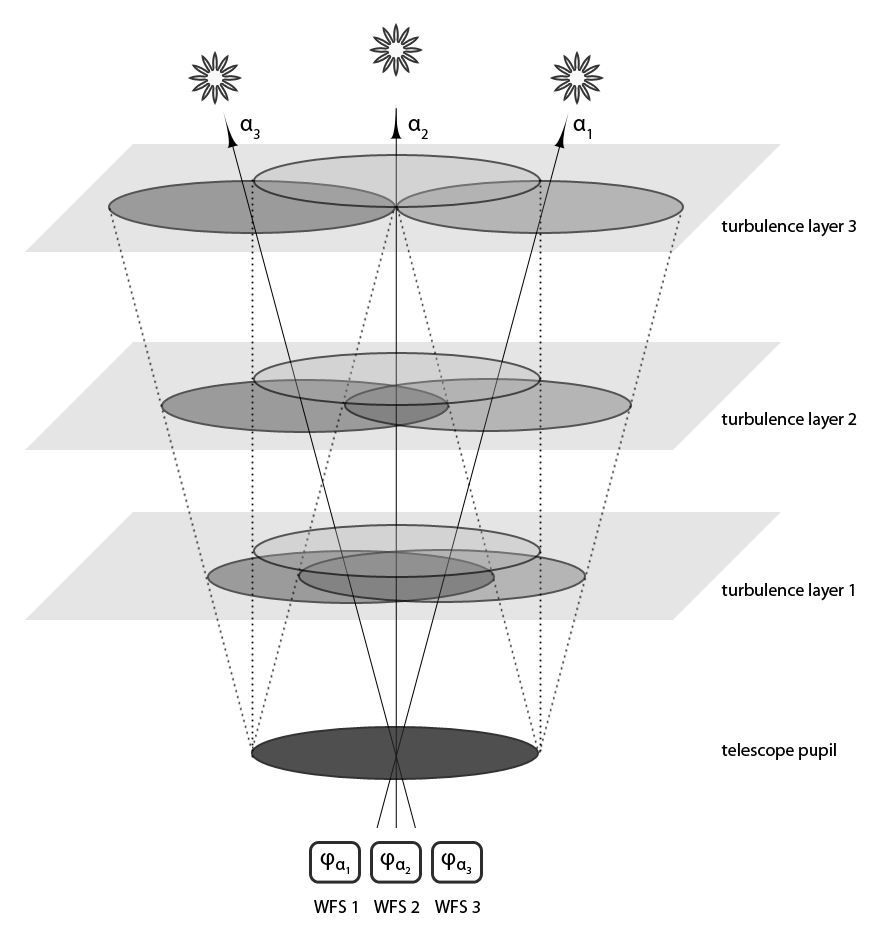}}
\caption{\label{atmfig}Layered atmosphere with three viewing directions. At each
layer, turbulence can only be reconstructed in the intersections of the shifted pupil areas.}
\efig

The Atmospheric Tomography operator describes a limited angle tomography problem based on a layered structure of the atmosphere. Additionally, only a finite number of directions of views $\ag$, $g=1,\dots ,G$, associated to the location of $G$ guide stars (natural or laser) and a wavefront sensor measuring the incoming wavefronts $\data{g}$ from the guide star, are available. We assume that the atmosphere contains $L$ layers, where each layer is a plane at height $h_l$ parallel to the telescope aperture $\OT$ (see Figure \ref{atmfig}).

A standard setting for the ELT that is also used in our numerical examples is the ESO standard atmosphere with $L=9$ layers sampled by $G=6$ viewing directions. This setup is, e.g., used for ELT simulations of the instrument {\it Multiconjugate adaptive Optics Relay For ELT Observations} (MORFEO). In the following, we will denote the turbulence at layer $l$ by $\layerj{l}$ and the incoming wavefronts in direction $\ag$ by $\data{g}$. Assuming geometric optics, the connection between the turbulent layers $\bPhi = (\layerj{1}, \dots, \layerj{L})^T$ and the incoming wavefronts $\bphi=(\data{g})_{g=1,\dots,G}$
is given by the atmospheric tomography operator $\bA$, see \cite{Fusco}:

\begin{eqnarray}
  \bA&:&\bigotimes_{l=1}^L L_2(\Ol) \to L_2(\OT)^G\\
  \label{AtmTomo}
  \bA\bPhi &:=&  \bphi = \left (\begin{array}{l}\data{1}\\
  \vdots\\\data{G}
  \end{array}\right )\\
\data{g}(\br)&=&\Ao{g}\bPhi := \sum_{l=1}^L \layer{l}{g}~, \br \in \OT~, g=1,\dots , G.
\end{eqnarray}

The sets $\Ol$ are defined as
\begin{eqnarray}\label{Omj}
\Ol&=&\bigcup_{g=1}^G \OT (h_l\al_g),\\
\label{OmT}
\OT (h_l\al_g)&:=&\{\br\in \mathbb R^2:\br-h_l\al_g\in \OT \} ~.
\end{eqnarray}

The set $\Ol$ denotes the area on the layer $l$ that contributes to the measurements of $\varphi$. This area grows with increasing hight, see Figure \ref{atmfig}.
Please note that this definition only covers the case where the incoming light
originates from NGS. For LGS, in particular the cone effect has to be considered,
which leads to a slightly different definition of the operator, see, e.g.,
\cite{RoRa13}. Nevertheless, all of the obtained results in the paper can be easily generalized from the NGS to the LGS or mixed case.

For our analysis we further need to specify the underlying function spaces: the cartesian product spaces are equipped with the usual inner products, respectively, i.e.,
\[ \spr{\phi,\psi}:=\suml\frac{1}{\gamma_l}\spr{\phi_l,\psi_l}_{L_2(\Ol)}\qquad
 \spr{\varphi,\psi}:=\sumg\spr{\varphi_g,\psi_g}_{L_2(\OT)}\,. \]

The weights $\gamma_l$ can be used to include the turbulence strength in layer $l$ into the reconstruction. In order to reconstruct the turbulence profile $\bPhi$, we need to solve for given $\data{1}, \dots , \data{g}, \dots , \data{G}$ equation \eqref{AtmTomo}, i.e., the system of equations

\begin{equation}\label{KaczmarzEq}
  \Ao{g}\bPhi = \data{g}, \hspace{1cm}g=1,\dots , G.
\end{equation}

Classical reconstruction methods for linear operator equations make frequently use of the adjoint of the operator, which depends on the operator itself as well as on the used inner product. We have the following result:

\begin{prop}
  The adjoint of the atmospheric tomography operator \\ $\bA:L_2(\Ol)^L \to L_2(\OT)^G$ is given by
  \begin{equation}\label{Astar}
    \bA^\ast \bphi = \sum_{g=1}^{G} \Ao{g}^\ast \data{g}
  \end{equation}
  with
  \begin{eqnarray}
  \label{Aad}
  \left (\Ao{g}^\ast \data{g}\right )(\br) &=& \left(
\begin{array}{c}
   \gamma_1 \cdot  \data{g}\left(\br - h_1\ag\right) \chi_{\OT(h_1\ag )}\left(\br\right )\\
   \vdots\\
    \gamma_l \cdot  \data{g}\left(\br - h_l\ag\right) \chi_{\OT(h_l\ag )}\left(\br\right )\\
    \vdots\\
     \gamma_L \cdot  \data{g}\left(\br - h_L\ag\right) \chi_{\OT(h_L\ag )}\left(\br\right )
\end{array}\right)
\end{eqnarray}
\end{prop}

A proof can be found in \cite{RaRo12,SaRa15}.  Additionally, we have

\begin{proposition}
The operator $\Ao{g}^\ast$ is inverted on its range by $\Ao{g}$, i.e.,
\begin{equation}
\label{ident}
\left (\Ao{g}\Ao{g}^\ast \psi\right )(\br) = \psi (\br), ~~\psi\in L_2 (\OT)~.
\end{equation}
\end{proposition}

See \cite{RaRo12} for a proof.

\section{Non-Uniqueness of the Atmospheric Tomography Operator}\label{nonuniqueness}
In this section we show that the tomography equation \eqref{AtmTomo} does usually not admits a unique equation. To achieve this results, we need some geometrical results based on the atmospheric tomography setup.
\subsection{Some geometrical considerations}
For what follows we need
\begin{assumption}\label{GeomAssump}
For the guide star directions $\alpha_g , g=1,\dots , G$ and the layer heights $h_l, l=1,\dots , L$  holds
\begin{eqnarray}
\alpha_j &\neq& \alpha_g \mbox{ for } j\neq g\\
0&<& h_1<\dots < h_l < \dots < h_L.
\end{eqnarray}
Specifically, we assume that
\begin{equation}
	\alpha_g=\left (\cos\theta_g, \sin \theta_g\right )
\end{equation}
with $\theta_g\in [0, 2\pi )$, and $0\le\theta_g < \theta_{g+1}$ for $g=1,\dots , G-1$, i.e., the directions are ordered clockwise.
\end{assumption}

Moreover, it will be convenient to define the overlap function $\omega_l$:\\

\begin{definition}
  For a layer $l$, $l=1,\dots , L$, we define the overlap function $\omega_l$ as
  \begin{equation}
    \omega_l (\br) := \sum_{g=1}^G \chi_{\OT(h_l\alpha_g)}(\br).
  \end{equation}
\end{definition}

The overlap function counts to how many of the areas $\OT(h_l\alpha_{g})$ a point $\br$ belongs.
 We immediately observe\\

 \begin{remark}
   For the overlap function it holds that
   \begin{enumerate}
     \item[(i)] $\omega_l(\br) \in \{0,1,\dots , G\}$, i.e., it only admits discrete values.
     \item[(ii)] $\omega_l(\br) = 0 \Longleftrightarrow \br\in \left (\bigcup_{g=1}^G\OT(h_l\alpha_g)\right )^\perp$
     \item[(iii)] $\omega_l(\br)=1  \Longleftrightarrow \exists \tilde{g}\mbox{ s.t. } \br\in \OT(h_l\alpha_{\tilde{g}})\setminus \bigcup_{\stackrel{g=1}{g\neq\tilde g}}^G\OT (h_l\alpha_g)$.
     \item[(iv)] $\omega_l(\br)=G  \Longleftrightarrow \br\in \bigcap_{g=1}^G\OT (h_l\alpha_g)\neq \emptyset$.
   \end{enumerate}
 \end{remark}
~\\
As we can observe from Figure \ref{atmfig}, the overlap of the sets $\OT (h_l\alpha_g)$ decreases for higher altitudes $h_l$. Accordingly, at a certain altitude $h_{disj}(\alpha)$ there is no overlap anymore, i.e., $\omega_l (\br)\in \{0,1\}$ for all $h_l\ge h_{disj}$.\\

\begin{proposition}\label{Single_ball}
  Under Assumption \ref{GeomAssump}
  and $G>1$ there exists $\br_l\in\Omega(h_l\alpha_g)$ and $\rho_l>0$ s.t. $\omega_l(\br)=1$ for all
  $\br\in B_{\rho_l}(\br_l)$.\\
\end{proposition}

The proof of the above Proposition can be done using standard arguments, e.g., from analytical geometry. For an illustration please see Figure \ref{fig:Ball}. The graphics also gives the idea for the proof: One takes three neighboring areas $\OT (h_l\alpha_g)$, and computes the intersection points of the circles of the two outer areas. Then we draw a line through the center of $\Omega_T$ and the intersection point of the two outer circles. The point $\br_l$ is then placed in the middle of the line between the intersection of the two outer circles and the intersection of the line with the middle circle.\\

\begin{proposition}\label{Prop_omega=1}
  Let $\br\in\OT$ and
  \begin{eqnarray}
    \tilde \br &=& \br+h_l\al_g\\
    \tilde{\tilde \br} &=& \br+ h_{l+1}\al_g.
  \end{eqnarray}
  If $\omega_l (\tilde \br ) = 1$, then also $\omega_{l+1} (\tilde{\tilde \br}) = 1$.\\
\end{proposition}
\bfig
\centering
\bpict(5,5)
\put(0,0){\includegraphics[width=5cm]{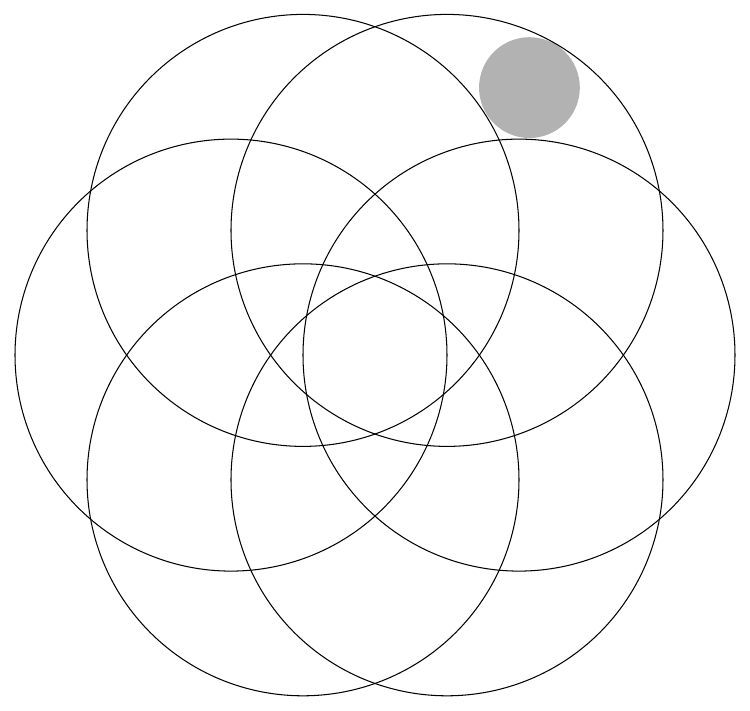}}
\put(4.9,4.1){\line(-2,0){1.3}}
\put(5,4){$B_{\rho_l}(\br_l)$}
%\put(.3,.3){$\OTt$}
\epict
\caption{A choice for the ball $B_{\rho_l}(\br_l)$ (dark grey) from Proposition \ref{Single_ball}. For this graphics, 6 different directions $\alpha_g$ have been used.}
\label{fig:Ball}
\efig
\begin{proof}
Due to their definition, $\tilde \br\in\OT (h_l\al_g)$ and $\tilde{\tilde \br}\in\OT (h_{l+1}\al_g)$. With $\omega_l (\tilde \br)=1$ it holds that $\tilde \br \notin \OT (h_l\al_{\tilde g})$ for all $\tilde g \neq g$. As $h_l\al_i$ is the center of the circle $\OT (h_l\al_i)$, we have in particular
\[\|\tilde \br - h_l\al_{\tilde g}\|> T \mbox{ for }\tilde g\neq g.\]
Showing $\omega_{l+1}(\tilde{\tilde \br})=1$ is equivalent to proving $\|\tilde{\tilde \br} - h_{l+1}\al_{\tilde g}\|>T$ for all $\tilde g\neq g$.
For simplicity of notation, we set $h=h_l$, $h_{l+1}=h+\Delta$, $\alpha:=\alpha_g$ and $\beta:=\alpha_{\tilde g}$. We thus obtain
\begin{eqnarray}\nonumber
  \|\tilde{\tilde \br}-h_{l+1}\alpha_{\tilde g}\|^2 &=& \|\br+(h+\Delta)\alpha-(h+\Delta)\beta\|^2\\\nonumber
  &=& \|\br+h(\alpha-\beta) +\Delta(\alpha-\beta)\|^2\\ \nonumber
  &=& \|\br+h(\alpha-\beta)\|^2+\Delta^2\|\alpha-\beta\|^2+2\Delta\langle \br+h(\alpha-\beta), \alpha-\beta\rangle\\
  &=& \|\br+h(\alpha-\beta)\|^2+(\Delta^2+2\Delta h)\|\alpha-\beta\|^2 +2\Delta\langle \br,\alpha-\beta\rangle. \hspace{1cm}\label{absch:1}
\end{eqnarray}
Additionally, we have
\[
  \|\tilde \br -h_l\alpha_{\tilde g}\|^2 = \|\br+h(\alpha-\beta)\|^2= \|\br\|^2+ h^2\|\alpha-\beta\|^2+2h\langle \br,\alpha-\beta\rangle > T^2
\]
or

\[
\langle \br,\alpha-\beta\rangle > \frac{T^2}{2h}-\frac{h}{2}\|\alpha-\beta\|^2-\frac{\|\br\|^2}{2h}
\]
Using this estimate we obtain with \eqref{absch:1}
\begin{eqnarray}\nonumber
  \|\tilde{\tilde \br}-h_{l+1}\alpha_{\tilde g}\|^2 &>& \|\br+h(\alpha-\beta)\|^2+(\Delta^2+2\Delta h)\|\alpha-\beta\|^2 +
  \frac{\Delta T^2}{h}-\Delta h\|\alpha-\beta\|^2-\frac{\Delta\|\br\|^2}{h}\\\nonumber
  &=& \|\br+h(\alpha-\beta)\|^2+(\Delta^2+\Delta h)\|\alpha-\beta\|^2 +\frac{\Delta T^2}{h}-\frac{\Delta\|\br\|^2}{h}.
\end{eqnarray}
Now as $\|\br\|\le T$ it follows $\frac{\Delta T^2}{h}-\frac{\Delta\|\br\|^2}{h}>0$, and with $\|\br+h(\alpha-\beta)\|=\|\tilde \br-h_l\alpha_{\tilde g}\|>T$ it follows
\[\|\tilde{\tilde \br}-h_l\alpha_{\tilde g}\|^2 > \|\br+h\alpha-h\beta\|^2+(\Delta^2+\Delta h)\|\alpha-\beta\|^2> T^2,\]
which concludes the proof.
\end{proof}

According to Proposition \ref{Single_ball}, for $h_1>0$ it exists
\[\br_1\in \OT (h_1\al_g) \mbox{ and } \rho_1> 0\]
such that
\begin{eqnarray}
  \mu (B_{\rho_1}(\br_1))&>&0\\
  B_{\rho_1}(\br_1)&\subset & \OT (h_1\al_g)\label{subsetB1}\\
  \omega_1 (\br) &=&1 \mbox{ for all }\br\in\OT (h_1\al_g).
\end{eqnarray}
Here, $\mu$ denotes the Lebesque measure.
Because of \eqref{subsetB1} there exists for all $\bar{\br}\in B_{\rho_1}(\br_1)$ an $\br\in \OT$ such that

\[\bar \br = \br+h_1\al_g.\]
For what follows, we define

\begin{eqnarray}\label{def:Omgtbar}
  \bOmgT &:=& \{\br\in \OT: \br+h_1\al_g\in B_{\rho_1}(\br_1)\}\subset \OT\\
  \bOmgl{g}{l} &:=& \{\br+h_l\al_g:\br\in \bar\Om^g_T\}
\end{eqnarray}
and obtain

\begin{proposition}\label{prop:ballsonlayers}
  With the above definitions we observe\\

  \begin{itemize}
    \item[(i)] \hspace{1cm}$\bOmgl{g}{l}\subset\OT (h_l\ag).$\\
    \item[(ii)] \hspace{1cm}$\bar\Omega^{g,1}=B_{\rho_1}(\br_1).$\\
    \item[(iii)] \hspace{1cm} $\forall \br \in \bOmgl{g}{l}: \omega_l (\br)=1$.
  \end{itemize}
\end{proposition}

\begin{proof}
(i) As $\bOmgT\subset\OT$, this follows directly from \eqref{def:Omgtbar}.\\
(ii) Follows also directly from the definition of $\bOmgl{g}{l}$ by setting $l=1$.\\
(iii) We know that for $\tilde \br\in B_{\rho_1}(\br_1)=\bar\Omega^{g,1}$ it holds that $\omega_1 (\tilde\br)=1$. Additionally, $\tilde \br = \br+h_1\al_g$ with $\br\in \bOmgl{g}{l}$, therefore
\begin{equation}
  \tilde{\tilde\br}= \br+h_2\al_g \in \bar\Omega^{2,g} \hspace{1cm}\forall \br \in\bOmgT.
\end{equation}

Setting $l=1$ in Proposition \ref{Prop_omega=1} yields $\omega_2(\br)=1$ for all $\br\in\bar\Omega^{g,2}$. Now the same argument can be repeated for $l=2,\dots , L-1$ and yields $\bar\Omega^{g,l+1}$, which finishes the proof.
\end{proof}

\subsection{Non-Uniqueness and Reconstructability}\label{reconstructability}
We have now collected all ingredients for a non-uniqueness result for the inversion of the atmospheric tomography operator.\\
\begin{proposition}
  Under the Assumptions \ref{GeomAssump} with $L>1$, the atmospheric tomography operator is not uniquely invertible.\\
\end{proposition}

\begin{proof}
  Let $\bPhi=(\layerj{1},\dots , \layerj{L})$ denote a given atmosphere, and $(\data{1},\dots , \data{G})$ the associated wavefronts in the telescope pupil, i.e.,
   \begin{equation}
     \data{g}(\br)=\sum_{l=1}^L \layer{l}{g}.
   \end{equation}
   We want to show that two different atmospheres create the same set of data. It is actually sufficient to show that two atmospheres create the same data $\data{g}$ for a fixed $g$. According to Proposition \ref{Single_ball} there exist $\br_1$ and $\rho_1$ s.t. $\omega_1(\br)=1$ for all $\br\in B_{\rho_1}(\br_1)$, and according to Proposition  \ref{prop:ballsonlayers} (iii) we have $\om_l(\br)=1$ for all $\br\in \bOmgl{g}{l}$.
   Let us restrict our attention to $\br\in\bOmgT$. Consequently, $\br+h_j\al_g\in\bOmgl{g}{l}$, and thus only function values of
    $\layerj{l}(\br)$ in $\bOmgl{g}{l}$ contribute to the function values of
    $ \varphi_{\alpha_g}(\br)$ in $\bar\Omega_T$. As $\omega_{l}(\br)=1$ for all $\br\in\bOmgl{g}{l}$,
     $\br\not\in\OT(h_l\al_{\tilde g})$ for $g\neq\tilde g$ and thus changes in $\layerj{l}$ on $\bOmgl{g}{l}$ leave the computation of $\data{\tilde g}$ unchanged. Next, for $l=1,\dots , L-1$ we choose
    
     \begin{eqnarray}\nonumber
       \tlayerjrec{l}(\br+h_l\al_g) &=& \left\{
       \begin{array}{c}
         \layerj{l}(\br+h_l\al_g)\mbox{ for } \br\in \OT\setminus\bOmgT\\
         \mbox{ arbitrarily for } \br\in \bOmgl{g}{l}
       \end{array}\right .
\hspace{0.1cm}l=1,\dots , L-1
       \\ && \label{changed_atmo}\\
       \tlayerjrec{L}(\br+h_l\al_g)&=&\left\{
       \begin{array}{c}
         \layerj{L}\mbox{ for } \br\in \OT\setminus\bOmgT\\
         \data{g}(\br)-\sum_{l=1}^{L-1}\tlayerjrec{l}(\br+h_l\al_g) \mbox{ for }   \br\in \bOmgl{g}{l}.
       \end{array}\right .  \nonumber
     \end{eqnarray}
Consider the related atmospheric tomography data
$$\tilde\varphi_{\al_{\tilde g}}(\br)= \sum_{l=1}^L \tlayerjrec{l}(\br+h_l\al_{\tilde g}).$$
As the atmosphere has only been changed on the sets $\bOmgl{g}{l}$ which are disjunct to $\OT(h_l\al_{\tilde g})$ for $\tilde g\neq g$ we observe $\data{\tilde g}=\tilde\varphi_{\al_{\tilde g}}$. Additionally, we have $\tilde\varphi_{\al_{g}}(\br)=\data{g}(\br)$ for $\br\in\OT\setminus\bOmgT$ and
\begin{eqnarray*}\nonumber
  \tilde\varphi_{\al_{g}}(\br)&=&\sum_{l=1}^L \tlayerjrec{l}(\br+h_l\al_{g})\\
  &&\\ \nonumber
  &=& \sum_{l=1}^{L-1}\tlayerjrec{l}(\br+h_l\al_{g})+\data{g}(\br)-\sum_{l=1}^{L-1}\tlayerjrec{l}(\br+h_l\al_g)\\&=&\data{g} (\br)\hspace{1cm}\br\in\OT
  \mbox{ for }\br\in\bOmgl{g}{l},
\end{eqnarray*}
i.e., we have shown that $\tilde\varphi_{\al_{\tilde g}}=\data{\tilde g}$
for $\tilde g=1,\dots ,G$ although
$\layerj{l}\neq\tlayerjrec{l}$ for
$l=1,\dots , L$, which concludes the proof.
\end{proof}
~\\
\begin{remark}
\begin{enumerate}
  \item[(i)]Please note that changing the atmosphere according to \eqref{changed_atmo}  may lead to jumps at the boundaries of $\bOmgl{g}{l}$. However, changing the atmosphere smoothly at the boundary of $\bOmgl{g}{l}$ can be done by using the same ideas as above, e.g., by changing the atmosphere as in \eqref{changed_atmo}, but on an inner subset of $\bOmgl{g}{l}$ and then constructing a smooth transition of the atmosphere from the boundary of the inner subset to the boundary of $\bOmgl{g}{l}$.
  \item[(ii)] The construction of two different atmospheres that produce the same data depends on the construction of the set $B_{\rho}(r_1)$ for a direction $\al_g$ with $\om_1(\br)=1$ for $\br\in B_{\rho}(r_1)$. In fact with the same arguments one can show that the atmosphere cannot be uniquely determind in  the larger areas were $\om_l(\br)=1$ holds.
  \item[(iii)] It is easy to see that for systems with large separation, i.e., where the guide stars have a larger distance from each other, the area with no overlap grows. Therefore it will be more difficult to obtain meaningful reconstructions of the atmosphere, which may be an argument against such systems.
\end{enumerate}
\end{remark}
~\\
In a next step we would like to characterize the solutions that can be reconstructed by standard regularization methods. Let us first review some standard results from regularization theory. For solving a linear equation $Ax=y$, the concept of the generalized solution $x^\dagger$ is frequently used.
If $P_{\cal{R}(A)}$ denotes the orthogonal projection on $\cal{R}(A)$, the range of $A$, then $x^\dagger$ can be characterized as the (unique) solution of
\[Ax = P_{\cal{R}(A)}y\]
with minimal norm, see, e.g., \cite{Louis89,Engl}. In paricular we have\\

\begin{proposition}(\cite{Louis89}, Prop. 3.1.1)\label{prop:generalized_solution}\\ The generalized solution
  $x^\dagger=A^\dagger y$ is the unique solution of the normal equation
\begin{equation}
A^\ast A x = A^\ast y
\end{equation}
in $\overline{\mathcal{R}(A^\ast)}= \mathcal{N}(A)^\perp$.
\end{proposition}
~\\

Consequently, $x^\dagger$ belongs to $\overline{\mathcal{R}(A^\ast)}$. Additionally, most of the standard regularization methods produce approximations to $x^\dagger$ that belong to $\mathcal{R}(A^\ast)$. This is easy to see for Landweber iteration, which is defined by
\begin{equation}
  x_{k+1}^\delta = x_k^\delta + \beta A^\ast (y^\delta -Ax_k^\delta).
\end{equation}
By induction we see that if the initial iterate $x_0$ belongs to $\mathcal{R}(A^\ast)$ - which is the case for the standard choice $x_0=0$ - then also $x_{k+1}\in\mathcal{R}(A^\ast)$.
The conjugate gradient method has a similiar structure, i.e., the update belongs to $\mathcal{R}(A^\ast)$ and thus also the computed approximation as long as the initial iterate also belongs to $\mathcal{R}(A^\ast)$. It is also straight forward to see that the minimizer of the Tikhonov functional
\begin{equation}
  x_\alpha^\delta = (A^\ast A +\alpha I )^{-1}A^\ast y^\delta,
\end{equation}
belongs to $\mathcal{R}(A^\ast)$.
Consequently the structure of $\mathcal{R}(A^\ast)$ has a significant impact on the reconstructions. As we will see,
the existence of the sets $\bOmgl{g}{l}$ results in a specific structure of the adjoint for the atmospheric tomography problem.\\

\begin{proposition}\label{Prop:adjointonBall}
  For $\varphi =(\data{1}, \dots , \data{G})$ the adjoint of the atmospheric tomography operator $\bA$ on the sets $\bOmgl{g}{l}$ is given by
  \begin{equation}\label{eq:adjoint_on_Ball}
    \left (\bA^\ast\varphi\right )^{(l)}(\br +h_l\al_g) = \gamma_l\data{g}(\br)  \hspace{1cm} \forall l=1,\dots , L \mbox{ and } \br\in\bOmgT.\
  \end{equation}
\end{proposition}
\begin{proof}
  According to \eqref{Astar}, \eqref{Aad}, the adjoint of $\bf A$ is given as
  \begin{equation*}
    \bA^\ast \bphi = \sum_{g=1}^{G} \Ao{g}^\ast \data{g}
  \end{equation*}
with
  \begin{eqnarray*}
  \left (\Ao{g}^\ast \data{g}\right )(\br) &=& \left(
\begin{array}{c}
   \gamma_1 \cdot  \data{g}\left(\br - h_1\ag\right) \chi_{\OT(h_1\ag )}\left(\br\right )\\
   \vdots\\
    \gamma_l \cdot  \data{g}\left(\br - h_l\ag\right) \chi_{\OT(h_l\ag )}\left(\br\right )\\
    \vdots\\
     \gamma_L \cdot  \data{g}\left(\br - h_L\ag\right) \chi_{\OT(h_L\ag )}\left(\br\right )
\end{array}\right).
\end{eqnarray*}
Now let $\br\in \bOmgT$. Then $\br+h_l\al_g\in\bOmgl{g}{l}$ and
\begin{equation}
  \left (\Ao{g}^\ast \data{g}\right )^{(l)}(r+h_l\al_g)=\gamma_l\data{g}(r+h_l\al_g-h_l\al_g)\chi_{\OT(h_l\ag )}\left(\br+h_l\al_g\right ) = \gamma_l\data{g}(r)
\end{equation}
as $r+h_l\al_g\in\OT(h_l\al_g)$ for $\br\in\Omega_T^{g}$. Moreover,
\begin{eqnarray}
  \left (\Ao{\tilde g}^\ast \data{g}\right )^{(l)}(r+h_l\al_g)=\gamma_l\data{g}(r+h_l(\al_g-\al_{\tilde g})\chi_{\OT(h_l\al_{\tilde g} )}\left(\br\right ) = 0,
\end{eqnarray}
as $\br+h_l\al_g\in\bOmgl{g}{l}$ and $\bOmgl{g}{l}\cap \OT(h_l\al_{\tilde g})=\emptyset$ for $g\neq\tilde{g}$. Therefore

\begin{equation}
  \left (\bf A^\ast\varphi\right )^{(l)}(\br+h_l\al) = \sum_{\tilde g=1}^G \left (\Ao{\tilde g}\data{\tilde g}\right )^{(l)}(\br+h_l\al)=\gamma_l\data{g}(r)
\end{equation}
which concludes the proof.
\end{proof}
~\\
Proposition \ref{Prop:adjointonBall} has a interesting consequence on the structure of the adjoint on the sets $\bOmgl{g}{l}$:\\

\begin{proposition}\label{prop:equality_on_Balls}
For $l,k\in 1,\dots , L$ and $\bphi = (\data{1},\dots , \data{G})$ holds
\begin{equation}\label{eq:equality_on_Balls}
  \left ( \bA^\ast \bphi\right )^{(l)}(\bOmgl{g}{l})=\frac{\gamma_l}{\gamma_k}\left ( \bA^\ast \bphi\right )^{(k)}(\bar\Omega^{g,k}).
\end{equation}
\end{proposition}
\begin{proof}
  It is
\begin{equation*}
  \br_l\in\bOmgl{g}{l}\Leftrightarrow \exists \br\in\bOmgT: \br_l=\br+h_l\al_g,
\end{equation*}
and therefore $\br_k=\br+h_k\al_g\in \bOmgl{g}{k}$, or more compact,
\begin{equation}
  \bOmgl{g}{l}=\bOmgT+h_l\al_g,
\end{equation}
i.e., the sets $\bOmgl{g}{l}$ are shifted versions of $\bOmgT$. Now, according to \eqref{eq:adjoint_on_Ball} we obtain for all $\br\in\bOmgT$
\begin{eqnarray*}
    \left (\bA^\ast\varphi\right )^{(l)}(\br +h_l\al_g) &=& \frac{\gamma_l\gamma_k}{\gamma_k}\data{g}(\br)\\
    &=&\frac{\gamma_l}{\gamma_k}\left (\bA^\ast\varphi\right )^{(k)}(\br +h_k\al_g)
\end{eqnarray*}
and therefore \eqref{eq:equality_on_Balls}.
\end{proof}
~\\

Now according to Proposition \ref{prop:generalized_solution} and the comments thereafter, the generalized solution of the problem
\[\bA\bPhi=\bphi =\left (\data{1},\dots,\data{G}\right )^T \]
belongs to the range of the adjoint of $\bA$, i.e., $\bPhi^\dagger\in\mathcal{R}(\bA^\ast)$, and the same holds for the regularized solutions $\bPhi^\delta_{reg}$ obtained for most of the popular regularization techniques, e.g., for iterative reconstructors or Tikhonov regularization. Thus we derive from Proposition \ref{prop:equality_on_Balls} the following\\

\begin{remark}\label{rem:RecsPerp}
The generalized solution $\bPhi^\dagger$ as well as regularized reconstructions that belong to ${\mathcal{R}(\bA^\ast)}$ have in the sets $\bOmgl{l}{g}$ - apart from a scaling factor - exactly the same values.
\end{remark}
~\\

Keeping in mind that the original turbulent layers $\layerj{l}, ~l=1,\dots , L,$ are random functions that are independent of each other, it seems unlikely that their function values in the sets $\bOmgl{l}{g}$ differ only by a factor. As a consequence, this would mean that we will only be able to reconstruct turbulences that are not appearing naturally.

\subsection{A remark on Landweber-Kaczmarz for Atmospheric Tomography}

The structure of the operator $\bA$, i.e., its description by a set of operator equations \eqref{KaczmarzEq} suggests the use of the Landweber - Kaczmarz method, see Algorithm \ref{KaczAlg}.

\begin{algorithm}
\caption{Kaczmarz iteration\label{KaczAlg}}
\begin{algorithmic}
\STATE Choose $\bPhi_0$
\FOR {$i=1,\dots $}
\STATE$\bPhi_{i,0}=\bPhi_{i-1}$
\FOR {$g=1, \dots, G$}
\STATE $\bPhi_{i,g}=\bPhi_{i,g-1}+\beta_g\cdot\Ao{g}^\ast \left (\data{g}-\Ao{g}\bPhi_{i,g-1}\right )$
\ENDFOR
\STATE $\bPhi_{i}=\bPhi_{i,G}$
\ENDFOR
\end{algorithmic}
\end{algorithm}

Unfortunately, it has been observed that the method applied to the atmospheric tomography problem does only converge if the scaling parameters $\beta_g$ are chosen in dependence on the iteration number and decrease towards zero, see \cite{RoRa13}. This behaviour may be explained by the following
\begin{proposition}
Let $\bPhi_{k,g-1}$ be arbitrarily given. Then
 \begin{equation}\label{KaczIter}
   \bPhi_{i,g}=\bPhi_{i,g-1}+\beta_g\cdot\Ao{g}^\ast \left (\data{g}-\Ao{g}\bPhi_{i,g-1}\right )
 \end{equation}
 solves the equation
 $$\Ao{g}\bPhi_{k,g}= \beta_g\cdot\data{g}+(1-\beta_g)\Ao{g}\bPhi_{k,g-1}$$.
\end{proposition}
\begin{proof}
Applying $\Ao{g}$ to both sides of \eqref{KaczIter} yields
\begin{eqnarray*}
  \Ao{g}\bPhi_{i,g}&=&\Ao{g}\bPhi_{i,g-1}+\beta_g\cdot\Ao{g}\Ao{g}^\ast \left (\data{g}-\Ao{g}\bPhi_{i,g-1}\right )\\
  &\stackrel{\eqref{ident}}{=}& \beta_{g}\data{g}+(1-\beta_g)\Ao{g}\bPhi_{i,g-1}.
\end{eqnarray*}
\end{proof}
\begin{remark}
The specific scaling parameter choice $\beta_g=1$ for all $g=1,\dots, G$ yields that $\bPhi_{i,g}$ solves
$$\Ao{g}\bPhi_{k,g}= \data{g}.$$
\end{remark}
For the Landweber-Kaczmarz method with scaling parameter $\beta_g=1$ this means that each new iterate $\bPhi_{k,g}$ solves the above equation exactly. If the atmospheric tomography equation $\bA \bPhi = \bphi$ has a solution $\bPhi^\dagger$, and ${\cal{N}}(\Ao{g})=\{0\}$ for $g=1,\dots , G$, then in particular $\Ao{1}\bPhi^\dagger=\data{1}$ and the first iterate produces the solution, regardless of the starting value. If, however, solutions of $\Ao{g}\bPhi=\data{g}$ do exist that are no solution to \eqref{AtmTomo}, then it might happen that the iteration jumps in a circular way between those solutions and the iteration does not converge at all with this scaling parameter choice.

% !TEX root = AtmTomoAnalysis.tex
\section{Numerics}\label{numerics}
Now let us give some numerical verification to our developed theory. It is a well known issue for MCAO systems that the quality of the AO correction degrades if we move further away from the center. As we will see, we obtain the best correction in the area with highest overlap, i.e., for $\omega_l(x)=G$ for l=1,\dots ,L and it drops dramatically with the number of overlaps.
For our verification we use a setup similar to the MORFEO instrument \cite{MORFEO} of the E-ELT. We simulate three turbulent layers and assume that the deformable mirrors are conjugated to the height of the layers. By doing so, we avoid the fitting step that projects the reconstructed layers onto the deformable mirrors. For the reconstruction process we use different numbers of {\it Shack Hartmann Wavefront Sensors (SH WFS)}. In this way we vary the amount of data, which impacts the reconstruction quality. The SH WFS consists of a quadratic array of small lenslets and a CCD photon detector lying behind this array. The vertical and horizontal shifts of the focal points determine the average slope of the wavefront over the area of the
lens, known as subaperture. For the parameters of the system and the simulation see Table \ref{tab:MorfeoParams} and \ref{tab:SimParam}.\\
\begin{table}[!ht]
\begin{minipage}[t]{6cm}\centering
\begin{tabular}{|c|c|}\hline
  Telescope diameter & 42m\\
  \hline
  Central obstruction & no\\
  \hline
  Outer Scale $L_0$ & 25m\\
  \hline
  Fried parameter & 15.7cm\\
  \hline
\end{tabular}
\end{minipage}
\begin{minipage}[t]{6cm}
 \begin{tabular}{|c|c|c|}\hline
   Layer & Height & Strength\\
   \hline\hline
  1 & 0m & 0.75\\
  \hline
  2 & 4000m & 0.15 \\
  \hline
  3 & 12700m & 0.1\\ \hline
 \end{tabular}
\end{minipage}
\caption{System parameters similar to the MORFEO instrument at the E-ELT and the simulated atmosphere.}\label{tab:MorfeoParams}
\end{table}

\begin{table}[!ht]
\begin{minipage}[t]{5cm}\centering
\begin{tabular}{|c|c|}\hline
  \multicolumn{2}{|c|}{WFS for NGS}\\
  \hline\hline
  \# & 2-12\\
  \hline
  Type & SH-WFS\\
  \hline
  Geometry & 84x84 Subap\\
  \hline
  \# of Photons & 10000\\
  Wavelength & 589nm\\
  \hline
\end{tabular}
\end{minipage}
\begin{minipage}[t]{6cm}
 \begin{tabular}{|c|c|c|c|}\hline
   DM &actuators &Height&spacing\\
   \hline\hline
  1 & 85x85 & 0m & 0.5m\\
  \hline
  2 & 47x47 &4000m & 1m \\
  \hline
  3 & 53x53 &12700m & 1m\\ \hline
 \end{tabular}
\end{minipage}
\caption{Simulation parameters of wavefront sensor and deformable mirrors}\label{tab:SimParam}
\end{table}
As quality criteria for the reconstruction quality we use both the $L_2$ error of the reconstructed layers and the Strehl ratio. The latter is defined as the ratio between the maximum of the real energy distribution of incoming light in the image plane $I(x,y)$ over the hypothetical distribution $I_D(x,y)$, which stems from the assumption of diffraction-limited imaging,
\begin{equation}
  \text{SR}:= \frac{\max_{(x,y)}I(x,y)}{\max_{x,y}I_D(x,y)}.
\end{equation}
By its definition, $\text{SR}\in [0,1]$; frequently it is also given in $\%$, i.e., $\text{SR}\in [0,100\%]$. A Strehl ration $\text{SR}=1 ~(100\%)$ means that the influence of the atmosphere has been removed from the observation, i.e., the observation is only diffraction limited. For its numerical evaluation, the Marechal criterion is used. For more details on the Strehl ratio we refer to \cite{Roddier}. The simulations of the forward evaluation have been carried out in the internal and entirely MATLAB-based AO simulation tool MOST, which has been developed by the Austrian Adaptive Optics (AAO) team. The reconstructions have been carried out using the FEWHA algorithm \cite{YuHeRa13b,StadlerRamlau2021,StadlerRamlau2022}. The evaluation of the Strehl ratios is done at a radially symmetric grid, see Figure \ref{fig:Strehlpoints}.
Typically, the SH WFS suffers from read-out noise and photon noise. The read-out noise is due to errors while reading photons via the charge-coupled device (CCD) detector planes. This kind of noise is measured in electrons per pixel and set to $3.0$ for the simulations. The photon noise is related to the number of photons that are sensed by the CCD in a subaperture during a certain time frame. For an NGS with a large number of photons the noise is commonly approximated by a Gaussian random variable with zero mean and covariance matrix $C_\eta$. The noise is identically distributed in each subaperture and the x- and y-measurements are uncorrelated. Hence, the covariance matrix is given by
\begin{equation}
C_\eta = \sigma^2 I,
\end{equation}
where $\sigma^2$ is the noise variance of a single measurement defined by $\sigma^2=1/n_{photons}$. In our simulations the number of photons is set to $10000$ for all NGS.

Figure \ref{fig:RecoLayer3} shows the original atmosphere on layer 3 as well as its reconstruction based on data from 6 NGS. Visual inspection shows that at least some of the structures in the centre of the atmosphere have been captured, which is not true for structures close to the borders.

\begin{figure}[!ht]
\centering
  \includegraphics[width=6.8cm]{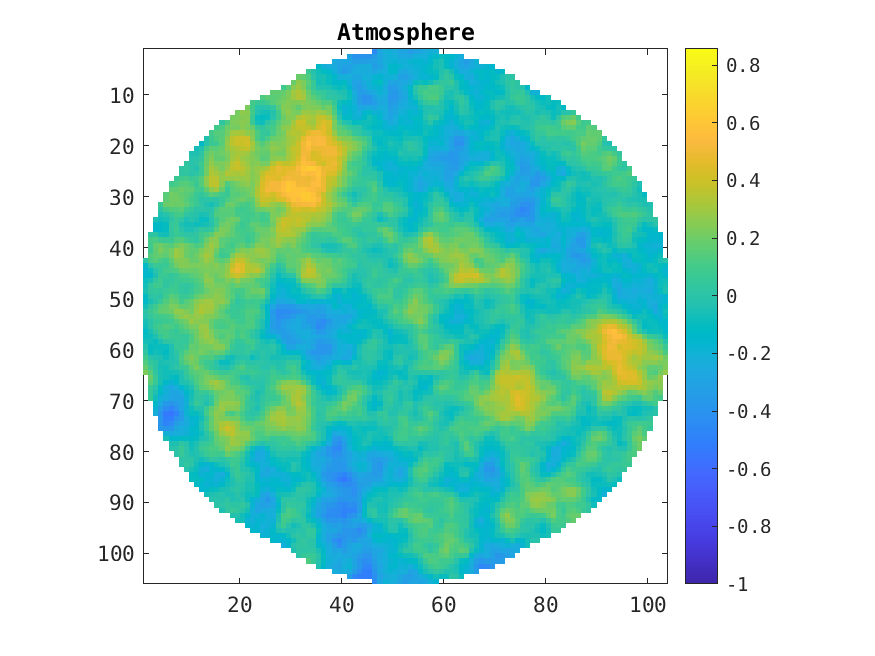}\hspace*{-0.8cm}
  \includegraphics[width=6.8cm]{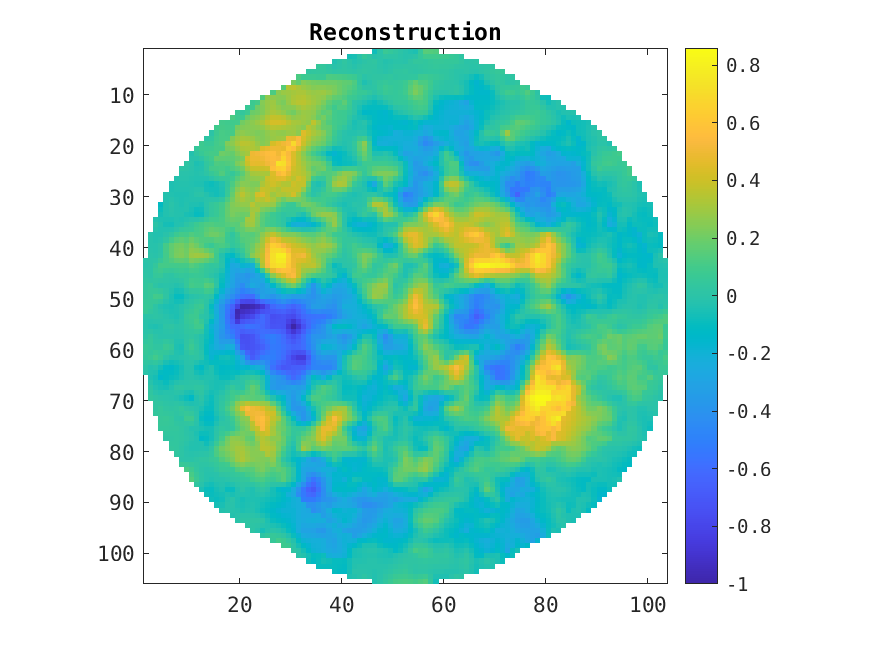}
  \caption{Original turbulence on layer 3 within $\Omega_3$ (left) and its reconstruction using 6 guide stars.}
  \label{fig:RecoLayer3}
\end{figure}

In Table \ref{table:Reconstruction_Results} we present the reconstruction results, i.e., the $L_2$ error and the $\text{SR}$, for 6 NGS. The relative $L_2$ error is calculated via
\begin{equation}
    \text{error} := \frac{\|{ \bPhi - \bPhi_{rec}}\|_2^2}{\|{\bPhi_{rec}}\|_2^2}.
\end{equation}
We observe that, although the reconstruction error is decreasing with increasing overlap, it still remains rather large. This is easily explained by the fact that the simulated atmosphere does not necessary belong to the orthogonal complement of the nullspace of the atmospheric tomography operator, whereas our reconstruction belongs to the orthogonal complement, see Remark \ref{rem:RecsPerp}. Therefore, the reconstruction error is bounded from below by the norm of its projection to the nullspace,
\[\|\bPhi_{rec}-\bPhi^\dagger\|\ge \|{\cal P_{N(\bA)}}\bPhi^\dagger\|.\]

\begin{table}[!ht]\centering
  \begin{tabular}{|c|c|c|c|c|c|c|}\hline
     &$\omega_l=1$&$\omega_l=2$& $\omega_l=3$&$\omega_l=4$&$\omega_l=5$&$\omega_l=6$\\
    \hline\hline
   $\text{error}$ & 11.72640 & 1.48710 & 1.06430 & 1.01750 & 0.99180 & 0.82200\\
   \hline
   SR & 0.00340 & 0.00640	& 0.01050 & 0.02400 & 0.03720	 & 0.40380\\ \hline
    SR variance & 0.00070 & 0.00160	& 0.00280 & 0.00670 & 0.01110	 & 0.17170	\\ \hline
  \end{tabular}
  \caption{$L_2$ error, Strehl Ratio (SR) and SR variance in directions with different numbers of overlaps on Layer 3.}
  \label{table:Reconstruction_Results}
\end{table}

\begin{figure}[H]
  \centering
  \includegraphics[width=9cm]{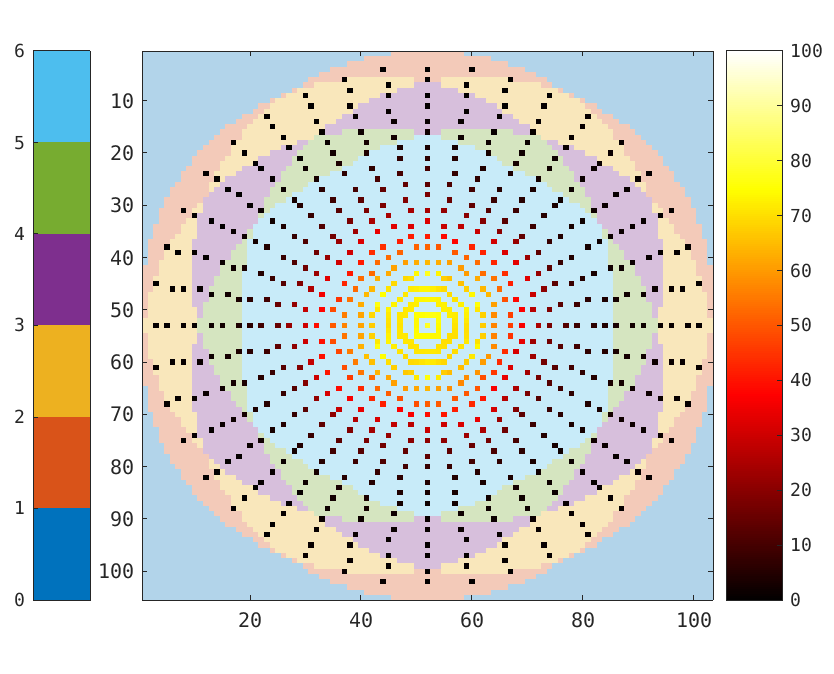}
  \caption{The area $\Omega_3$ on layer 3. The values of $\omega_3(x)\in \{1,\dots , 6\}$ are color coded, see colorbar on the left side of the picture. The dots mark the locations where the Strehl ratio is evaluated. The Strehl ration achieved at that location are alos color coded, see colorbar on the right side of the picture.}
  \label{fig:Strehlpoints}
\end{figure}
Although the reconstruction accuracy is limited, the achievable correction, given by the Strehl ratio, increases significantly and reaches a value of about $0.41$ on the third layer with maximal overlap. We also observe that there is quite a significant jump in the average Strehl Ratio from the area of 5 overlaps to the area of 6 overlaps. This is easily explained by looking at the achieved Strehl ratios within the area of $6$ overlaps: while we have in the center Strehl ratios of more than $80\%$, they drop significantly towards the outer borders of the area, to about $3\%$, see Figure \ref{fig:Strehlpoints}. Additionally, this also can be seen by looking at the variance in Strehl Ratio within areas of same overlaps (last line in Table \ref{table:Reconstruction_Results}): The variance for the area with $\omega_l=6$ is by a factor of at least 20 larger than the variances for the other areas. We get similar results with different numbers of guide stars. Note that the average Strehl ratio in the areas with maximal overlaps increases with the number of guide stars, as can be seen in Table \ref{table:Strehl_with_different_numberofguidestars}:

\begin{table}[!ht]\centering
  \begin{tabular}{|c|c|c|c|c|}\hline
     &3 GS& 6 GS &9 GS & 12 GS\\
    \hline\hline
    Average SR& 0.049 & 0.3209 & 0.3418& 0.3514\\
   \hline
   SR Layer 3 & 0.0624 & 0.40380 & 0.4374 & 0.4529\\ \hline
   Center SR & 0.1453 & 0.71210 & 0.7521 & 0.7632 \\ \hline
  \end{tabular}
  \caption{Strehl ratios in the areas of maximal overlaps for systems with different numbers of guide stars.}
  \label{table:Strehl_with_different_numberofguidestars}
\end{table}

In conclusion, we have seen that, although we are unable to reconstruct the correct atmosphere from the given guide star measurements, the achievable {\it correction} - given by the Strehl Ratio - is high and increases with the number of used guide stars.

Above we have stated that the the rather large reconstruction error for the turbulence in Table~\ref{table:Reconstruction_Results} can be explained by the fact that the simulated atmosphere does not necessarily belong to the orthogonal complement of the nullspace of the atmospheric tomography operator. We now validate this assertion via numerical simulations. To this end, we project the originally simulated atmosphere to the nullspace of the atmospheric tomography operator before reconstructing it. Using ${\cal N}(\bA)^\perp = {\cal R}(\bA^\ast)$, we create a projected atmosphere $\hat{\bPhi}$ by
\begin{equation}
    \hat{\bPhi} = \bA^*\bA \bPhi \in {\cal N}(\bA)^\perp.
\end{equation}
Figure~\ref{fig:ProjAtmo3} shows the projected turbulence and its reconstruction. One easily sees that the difference between the projected $\hat{\bPhi}$ here and the original atmosphere ${\bPhi}$ in Figure~\ref{fig:RecoLayer3} is large and that the projected one does not represent realistic atmospheric conditions. On the other hand, we observe that the reconstruction of the projected atmosphere is now much better than in the original setting in Figure~\ref{fig:RecoLayer3}. In regions with high overlaps we barely see any differences between the turbulence and its reconstruction, whereas for areas close to the border with a small number of overlaps the reconstruction is less good. This becomes evident when looking at the relative $L_2$ error in Table~\ref{table:Projected_Reconstruction_Results}. In general, the reconstruction error is much smaller compared to the original setting in  Table~\ref{table:Reconstruction_Results}. Further, the error is significantly decreasing for areas with a high number of overlaps.

\begin{figure}[H]
\centering
  \includegraphics[width=6.8cm]{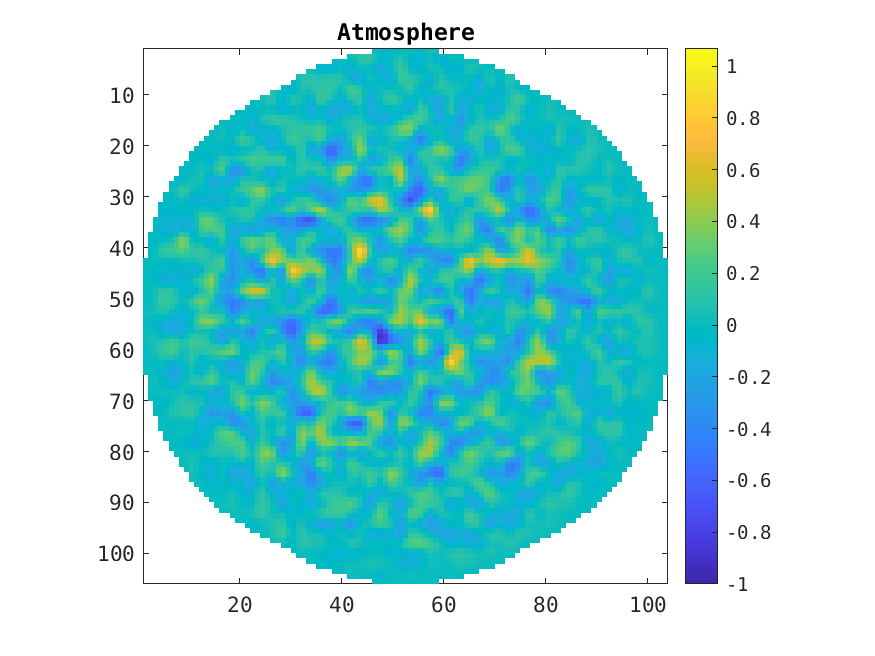}\hspace*{-1cm}
  \includegraphics[width=6.8cm]{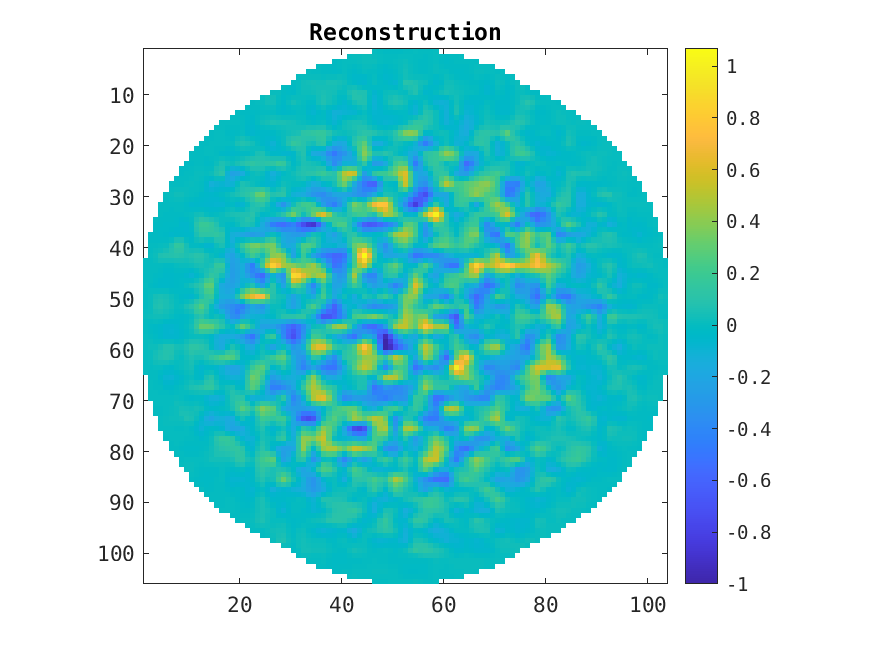}
  \caption{Projected turbulence on layer 3 within $\Omega_3$ (left) and its reconstruction with 6 guide stars.}
  \label{fig:ProjAtmo3}
\end{figure}

\begin{table}[h!]\centering
  \begin{tabular}{|c|c|c|c|c|c|c|}\hline
     &$\omega_l=1$&$\omega_l=2$& $\omega_l=3$&$\omega_l=4$&$\omega_l=5$&$\omega_l=6$\\
    \hline\hline
   $\text{error}$ & 9.90460 & 3.17020 & 1.23490 & 0.49410 & 0.30230 & 0.09400\\\hline
  \end{tabular}
  \caption{$L_2$ $error$ for areas with different numbers of overlaps.}
\label{table:Projected_Reconstruction_Results}
\end{table}

\section{Conclusion}
In the paper we analyzed the Atmospheric Tomography operator in particular with regard to its invertibility. It turns out that especially in non-overlapping areas of the layers there is not enough information in the data in order to achieve a correct reconstruction, which is also confirmed by the numerical simulations. Additionally, the numerical results indicate that in regions with many overlaps a fully correct reconstruction seems impossible. Nevertheless, the AO-correction achieved by the non-optimal reconstruction is still very good in terms of Strehl ratio, which can be explained by the fact that for the correction only the sum of the turbulence along lines is important.\\
The above results show that the nullspace of the Atmospheric Tomography is rather large, and that standard reconstruction/regularization schemes will almost never succeed in providing an optimal physical reconstruction. If in contrast the atmosphere would belong to the range of the Atmospheric Tomography operator, a reconstruction of the turbulence above the telescope would be possible, as demonstrated by numerical experiments.\\
The non-uniqueness issues are especially pronounced in the non-overlapping areas. This suggests in particular problems for tomography systems with large angular separation.

\bibliographystyle{siam}
\bibliography{Refs_Ramlau.bib}

\end{document}